\documentclass[12pt]{amsart}

\usepackage[latin1]{inputenc}
\usepackage{latexsym, enumerate}
\usepackage{graphicx}
\usepackage{subfigure}
\usepackage{color}
\usepackage{amsmath, amsfonts, amssymb, hyperref,stmaryrd}
\usepackage{graphics}

\newcommand{\N}{\mathbb{N}}

\newcommand{\fmf}{$F$-mul\-ti\-pli\-city free}
\definecolor{gr75}{gray}{0.75}

\numberwithin{equation}{section}
\newcommand{\news}[1]{\ensuremath{\blacktriangleright _s}#1\ensuremath{\blacktriangleleft _s}}

\renewcommand{\news}[1]{#1}

\newcommand{\ms}{\begin{math}}
\newcommand{\me}{\end{math}}

\newcommand{\qs}{\mathcal{S}}

\newcommand{\LC}{\mathcal{L}_C}
\newcommand{\QQ}{QSym}
\newcommand{\SSS}{Sym}

\newcommand{\com}{\operatorname{com}}
\newcommand{\des}{\operatorname{des}}
\newcommand{\set}{\operatorname{set}}
\newcommand{\sh}{\operatorname{sh}}

\newlength\cellsize 
\savebox2{%
\begin{picture}(15,15)
\put(0,0){\line(1,0){15}}
\put(0,0){\line(0,1){15}}
\put(15,0){\line(0,1){15}}
\put(0,15){\line(1,0){15}}
\end{picture}}
\newcommand\cellify[1]{\def\thearg{#1}\def\nothing{}%
\ifx\thearg\nothing
\vrule width0pt height\cellsize depth0pt\else
\hbox to 0pt{\usebox2\hss}\fi%
\vbox to 15\unitlength{
\vss
\hbox to 15\unitlength{\hss$#1$\hss}
\vss}}
\newcommand\tableau[1]{\vtop{\let\\=\cr
\setlength\baselineskip{-16000pt}
\setlength\lineskiplimit{16000pt}
\setlength\lineskip{0pt}
\halign{&\cellify{##}\cr#1\crcr}}}
\savebox3{%
\begin{picture}(15,15)
\put(0,0){\line(1,0){15}}
\put(0,0){\line(0,1){15}}
\put(15,0){\line(0,1){15}}
\put(0,15){\line(1,0){15}}
\end{picture}}
\newcommand\expath[1]{%
\hbox to 0pt{\usebox3\hss}%
\vbox to 15\unitlength{
\vss
\hbox to 15\unitlength{\hss$#1$\hss}
\vss}}

\theoremstyle{plain}
\newtheorem{theorem}{Theorem}[section]

\newtheorem{lemma}[theorem]{Lemma}
\newtheorem{corollary}[theorem]{Corollary}

\newtheorem{definition}[theorem]{Definition}
\newtheorem{example}[theorem]{Example}

\newtheorem{remark}[theorem]{Remark}

\begin{document}
\title[Multiplicity free quasisymmetric and Schur functions]{Multiplicity free Schur, skew Schur, and quasisymmetric Schur functions}

\author{C.\ Bessenrodt}
\address{Institut f\"{u}r Algebra, Zahlentheorie und Diskrete Mathematik,
Leibniz Universit\"{a}t, Hannover,
D-30167, Germany}
\email{\href{mailto:bessen@math.uni-hannover.de }{bessen@math.uni-hannover.de}}

\author{S.\ van Willigenburg}
\address{Department of Mathematics, University of British Columbia, Vancouver, BC V6T 1Z2, Canada}
\email{\href{mailto:steph@math.ubc.ca}{steph@math.ubc.ca}}
\thanks{The authors' collaboration was supported in part by the Alexander von Humboldt Foundation
and the National Sciences and Engineering Research Council of Canada.}
\subjclass[2010]{Primary 05E05; Secondary 05A15, 05E15}
\keywords{compositions, multiplicity free, quasisymmetric function, Schur function, skew Schur function, tableaux}

\begin{abstract}
In this paper we classify all Schur functions and skew Schur functions that are multiplicity free when expanded in the basis of fundamental quasisymmetric functions, termed \fmf.  Combinatorially, this is equivalent to classifying all skew shapes whose standard Young tableaux have distinct descent sets. We then generalize our setting, and classify all \fmf\ quasisymmetric Schur functions with one or two terms in the expansion, or one or two parts in the indexing composition. This identifies composition shapes such that all standard composition tableaux of that shape have distinct descent sets.
We conclude by providing such a classification for quasisymmetric Schur function families, giving  a classification of Schur functions that are in some sense almost \fmf.
\end{abstract}

\date{Version of February 27, 2012}

\maketitle

\emph{Running title: Multiplicity free quasisymmetric and Schur functions}

\section{Introduction}\label{sec:intro} The algebra of symmetric functions, $Sym$, is central to algebraic combinatorics and impacts many other areas of mathematics such as representation theory and algebraic geometry. A nonsymmetric generalization of $Sym$ is the algebra of quasisymmetric functions, $\QQ$. $\QQ$ was first introduced in the 1980s as weight enumerators for P-partitions~\cite{Gessel}. Since then quasisymmetric functions have grown in importance, arising in many areas including the theory of Hopf algebras~\cite{MR}, chain enumeration \cite{Ehrenborg}, discrete geometry via the cd-index~\cite{BHvW}, combinatorics of random walks \cite{Hersh-Hsiao}, representation theory~\cite{Hivert}, and affine Grassmannians~\cite{LLMS}.

A natural question to study in symmetric function theory is when certain expansions of symmetric functions are multiplicity free. For example, the product of two Schur functions when expanded as a sum of Schur functions~\cite{Stembridge}, or two Schur $P$-functions expanded as a sum of Schur $P$-functions \cite{Bessenrodt}, skew Schur functions expanded as a sum of Schur functions~\cite{Gutschwager, ThomasYong}, Schur $P$-functions as a sum of Schur functions~\cite{ShawvW}.

In this paper we extend this study and consider the problem of when a (skew) Schur function expanded as a sum of fundamental quasisymmetric functions is multiplicity free, termed \fmf. Fundamental quasisymmetric functions are important to consider due to their connections to representation theory, poset enumeration and permutation enumeration. In our context, the property of a skew Schur function indexed by a shape~$D$
being \fmf\ is equivalent to the property that the set of  standard Young tableaux  of shape $D$ consists of tableaux that all have distinct descent sets.
We then generalize this question to consider recently discovered quasisymmetric functions that inherit many properties of Schur functions \cite{SQS, QS, LRQS,  Lauve-Mason}, known as quasisymmetric Schur functions.

More precisely, this paper is structured as follows. In Section~\ref{sec:background} we review necessary background material. Then in Section~\ref{sec:FmfSchur} we classify all Schur functions and skew Schur functions that are \fmf\ in Theorem~\ref{the:SFmf} and Theorem~\ref{the:skewFmf}, respectively. 
We focus our attention on quasisymmetric Schur functions in 
Section~\ref{sec:QSFmf12terms}, classifying all those with one or two terms and observing they are \fmf\ in Theorem~\ref{the:QS12term}. We classify \fmf\ quasisymmetric Schur functions whose indexing composition has  two parts in Theorem~\ref{the:QSFmf12part}. Finally in Theorem~\ref{the:Fmffam} we classify all families of quasisymmetric Schur functions each member of which is \fmf.

\

\noindent\textbf{Acknowledgements.} The authors would like to thank Sarah Mason for her quasisymmetric Schur function generating program. They would also like to thank Caleb Cheek and the anonymous referees for helpful suggestions.

\section{Background and preliminaries}\label{sec:background}
\subsection{Compositions and partitions}\label{subsec:comppart}
A \emph{composition} $\alpha = (\alpha _1, \ldots , \alpha _k)$ of $n$, denoted $\alpha \vDash n$, is a list of positive integers such that
$\sum _{i=1}^k \alpha_i=n$. We call the $\alpha _i$ the \emph{parts} of $\alpha$ and $n$ the \emph{size} of $\alpha$, denoted $|\alpha |$. The number of parts, $k$, is called the \emph{length} of $\alpha$ denoted $\ell(\alpha)$, and the magnitude of the largest part is called the \emph{width} of $\alpha$ denoted $w(\alpha)$. If $\alpha _j = \cdots = \alpha _{j+m}= i$ we often abbreviate this sublist to $i^m$. A \emph{partition} is a composition whose parts are in weakly decreasing order, that is $\alpha _1 \geq \cdots \geq \alpha _k$. Note that every composition $\alpha$ determines a partition $\lambda(\alpha)$ that is obtained by arranging the parts of $\alpha$ in weakly decreasing order. We denote by $\langle\lambda \rangle  $ the set of all compositions that determine the partition $\lambda$.

Given a composition  $\alpha$,
there are two compositions that are closely related to $\alpha$.
The first is the reverse of $\alpha$, $\alpha ^\ast = (\alpha _k, \ldots , \alpha _1)$, and the second is the complement of $\alpha$. To define the complement of $\alpha$, $\hat{\alpha}$,  recall the natural bijection between compositions of $n$ and subsets of $[n-1]=\{1, 2, \ldots , n-1\}$:
$$\alpha = (\alpha _1, \ldots , \alpha _k) \leftrightarrow \{ \alpha _1 , \alpha _1+\alpha _2 , \ldots , \alpha _1 + \alpha _2 +\cdots + \alpha _{k-1}\}=\set(\alpha).$$
Then $\hat{\alpha}=\set^{-1}((\set(\alpha)^c)$.
Note that forming the reverse and the complement commutes, that is,
$\hat{\alpha}^*=\widehat{\alpha^*}$.

Given compositions $\alpha, \beta$, we say that $\alpha$ is a \emph{refinement} of $\beta$ (or $\beta$ is a \emph{coarsening} of~$\alpha$), denoted $\alpha \preccurlyeq \beta$, if summing some consecutive parts of~$\alpha$ gives $\beta$. We also define the \emph{concatenation} of $\alpha$ and $\beta$, denoted $\alpha \cdot \beta$, to be the composition consisting of the parts of~$\alpha$ followed by the parts of~$\beta$. One final special composition is the empty partition of~0,
denoted~$\emptyset$.

\begin{example}If $\alpha = (2,1,2,2)$ and $\beta = (3,4)$ then $\lambda = \lambda(\alpha)=(2,2,2,1)$
$$\langle\lambda\rangle   = \{(1,2,2,2), (2,1,2,2), (2,2,1,2), (2,2,2,1)\}.$$Note that $\alpha^\ast = (2,2,1,2)$, $\hat{\alpha}= (1,3,2,1)$, $\alpha \preccurlyeq \beta$ and $\alpha \cdot \beta = (2,1,2,2,3,4)$.
\end{example}

\subsection{Young diagrams and tableaux}\label{subsec:SYT}
Given a partition $\lambda = (\lambda _1, \ldots , \lambda _{\ell(\lambda)})$
we say that its \emph{(Young) diagram}, also denoted $\lambda$, is the array
of left-justified cells with $\lambda _i$  cells in row~$i$, from the top,
for $1\leq i \leq \ell(\lambda)$.
We locate cells in the diagram by their row and column indices $(i,j)$
where $1\leq i \leq \ell(\lambda)$ and $1\leq j \leq \lambda _1$.
Given two partitions $\lambda$ and $\mu$ we say $\mu$ is \emph{contained}
in $\lambda$, denoted $\mu \subset \lambda$ if $\ell(\mu) \leq \ell (\lambda)$
and $\mu _i \leq \lambda _i$ for $1\leq i \leq \ell(\mu)$.
If $\mu \subset \lambda$ then the \emph{skew diagram} $D=\lambda /\mu$
is the array of cells whose indices satisfy
$$\lambda / \mu = \{(i,j)\ |\ (i,j)\in \lambda , (i,j) \not\in \mu\}.$$
Note that every diagram $\lambda$ can be considered as a skew diagram
$\lambda / \emptyset$. The number of cells in a diagram $D$ 
is called its \emph{size}, denoted $|D|$. 
The \emph{row lengths} of $D=\lambda / \mu$ are the numbers $\lambda_i-\mu_i$,
$1\leq i \leq  \ell(\lambda)$ \news{ and $\mu _i =0$ for $\ell(\mu) < i \leq \ell(\lambda)$};
similarly for the \emph{column lengths} of $D$ we count the \news{cells}
of $D$ in the columns. 
Given two  skew diagrams, $D_1$ and $D_2$, a disjoint union of them $D_1\oplus D_2$ is obtained by placing $D_2$ strictly to the north and east of $D_1$ in such a way that $D_1, D_2$ occupy none of the same rows or columns. We say a skew diagram is \emph{disconnected} if it can be written as a disjoint union of two skew diagrams, and \emph{connected} if it cannot.
Two useful operations on a skew diagram $D$ are the \emph{antipodal rotation} of $180^\circ$ in the plane of $D$, denoted $D^\circ$, and the \emph{transpose} of $D$, denoted $D^t$, where $D^t$ is the array of cells whose indices satisfy
$$D^t = \{ (i,j)\ |\ (j,i)\in D\}.$$If $D$ is a Young diagram with corresponding partition $\lambda$ then we call the partition $\lambda ^t$ corresponding to $D^t$ the \emph{transpose} of $\lambda$.

\begin{example} $$\lambda = (3,2,2,1) = \tableau{\ &\ &\ \\\ &\ \\\ &\ \\\ }\ \mu = (1,1) = \tableau{\ \\\ }\ \lambda / \mu = \tableau{&\ &\ \\& \ \\\ &\ \\\ }\ (\lambda/\mu)^t = \tableau{&&\ &\ \\\ &\ &\ \\\ }$$
\end{example}

A filling of the cells of a (skew) diagram $D$ with positive integers such that the entries weakly increase when read from left to right in a row, and strictly increase when read from top to bottom in a column, is called a \emph{semi-standard Young tableau}  (SSYT, plural SSYTx) $T$ of \emph{shape} $D$.
If furthermore $|D|=n$ and each of $1, \ldots ,n$ appears, then the filling is called a \emph{standard Young tableau} (SYT, plural SYTx).
Given an SYT $T$ we define its \emph{descent set} to be the set of all $i$ such that $i+1$ appears in a cell in a lower row than $i$.
Since we will also consider different descent sets in the context of compositions below, we denote the descent set of an SYT by $\des _p(T)$.
The composition $\set^{-1}(\des _p(T))$ associated to the descent set of
an SYT $T$ will be denoted by $\com(T)$.
Given an SSYT $T$ we say it is a \emph{Littlewood-Richardson} (LR) tableau if 
its reverse reading word is a lattice permutation, i.e., 
as we read the entries right to left by row from top to bottom,
then the number of $i$'s we have read is always at least the number of $i+1$'s.

\begin{example} For convenience we replace the cells by their entries when drawing tableaux:
$$\begin{matrix}&2&3\\&4\\1&6\\5\end{matrix}\qquad \begin{matrix}&1&1\\&2\\1&3\\2\end{matrix}$$
are respectively an SYT with descent set $\{ 3,4\}$ and an SSYT that is an LR tableau.
\end{example}

\subsection{Composition diagrams and tableaux}\label{subsec:SCT} Given a composition $\alpha = (\alpha _1 , \ldots , \alpha _{\ell(\alpha)})$ we say its composition diagram, also denoted $\alpha$, is the array of left-justified cells with $\alpha _i$  cells in row $i$ from the top, for $1\leq i\leq \ell(\alpha)$. As with Young diagrams, cells are located by their row and column indices, and the number of cells is called the \emph{size} of $\alpha$, denoted $|\alpha|$. Also, as with Young diagrams, we will want to create tableaux, and for our subsequent proofs the most useful definition is via the following poset first introduced in \cite{SQS}.

\begin{definition}\label{def:CT}
We say that the composition  \emph{$\gamma$ covers $\beta$}, denoted $\beta\lessdot \gamma $, if $\gamma $ can be obtained from $\beta $ either by prepending $\beta $ with a new part of size~1  
(i.e., $\gamma = (1)\cdot \beta$), 
or by adding 1 to the first (leftmost) part of $\beta$ of size $k$ for some $k$.
The partial order $\leq$ defined on the set of all compositions is the transitive closure of these cover relations, and the resulting poset we denote $\LC$.
\end{definition}

\begin{example}
The composition $\beta=(2,2,1,3,2,3)$ is covered in $\LC$ by the compositions
$(1,2,2,1,3,2,3)$, $(3,2,1,3,2,3)$,  $(2,2,2,3,2,3)$ and  $(2,2,1,4,2,3)$.
\end{example}

Standard composition tableaux can be created in the following way, which we can take
to be our definition by \cite[Proposition 2.11]{SQS}.

\begin{definition}\label{lem:makeCTfromLC} Let $\alpha \vDash n$ and
$$\emptyset= 
\alpha^{n+1} \lessdot \alpha ^{n} \lessdot \alpha ^{n-1} \lessdot
\cdots \lessdot \alpha ^{2} \lessdot \alpha ^1  = \alpha$$
be a sequence of consecutive cover relations in $\LC$.
If we fill the cell that differs between 
$\alpha ^{i+1}$ and $\alpha ^{i}$ with $i$ 
for
$1\leq i \leq n$ then we call the resulting filling a \emph{standard composition
tableau} (SCT, plural SCTx) of \emph{shape} $\alpha$.
\end{definition}

\begin{example}
The sequence
$$\emptyset \lessdot (1) \lessdot (1,1) \lessdot (2,1)
\lessdot (1,2,1) \lessdot (2,2,1) \lessdot (1,2,2,1) \lessdot
(1,3,2,1) \lessdot (2,3,2,1) $$
is a sequence of consecutive cover relations in $\LC$ that
corresponds to the SCT $T'$ in Example~\ref{ex:SCT} below.
\end{example}

As with SYTx we can define the descent set of an SCT. However in this case we define
the \emph{descent set} of an SCT $T$, denoted $\des _c(T)$, to be the set of all $i$
such that $i+1$ appears in a cell weakly
to the right of~$i$. It is straightforward to see
that given a composition diagram
$\alpha = (\alpha _1, \ldots , \alpha _{\ell (\alpha)})$,
if the cells of row $i$ are filled with
$$1+\sum _{j=0}^{i-1} \alpha _{j} , 
2+\sum _{j=0}^{i-1} \alpha _{j} , \ldots ,
\alpha_i -1 +\sum _{j=0}^{i-1} \alpha _{j} , 
\sum _{j=0}^{i} \alpha _{j}$$
(where $\alpha_0=0$) to form a filling $T$,
then $T$ is an SCT of shape~$\alpha$ with $\des _c(T) = \set(\alpha)$.
We call this filling the \emph{canonical  filling} of $\alpha$
and say that each row is \emph{row filled}.
The composition $\set^{-1}(\des_c(T))$ associated to the descent set of an
SCT $T$ will be denoted by $\com(T)$.
Lastly, given an SCT $T$ and a positive integer $m$,
we denote by $T+m$ the filling that has $m$ added to each entry of $T$.

\begin{example}\label{ex:SCT} The composition tableau 
$$T=\begin{matrix} 2&1\\5&4&3\\7&6\\8\end{matrix}$$
is the canonical filling of $(2,3,2,1)$ and
$$T'=\begin{matrix} 3&1\\5&4&2\\7&6\\8\end{matrix}$$
is an SCT with $des_c(T')=\{1,3,5,7\}$.
\end{example}

\subsection{Quasisymmetric and symmetric functions}\label{subsec:QsymSym}
The algebra of quasisymmetric functions, $\QQ$, is a graded algebra
$$\QQ := \QQ_0 \oplus \QQ _1 \oplus \cdots \subseteq \mathbb{Q}[ x_1, x_2, \ldots ]$$
where $\QQ _0$ is spanned by $M_0=1$ and all other $\QQ _n$ are spanned by
$\{ M_\alpha \} _{\alpha = (\alpha _1, \ldots , \alpha _{\ell(\alpha)})\vDash n}$
where
$$M_\alpha = \sum _{i_1<\cdots < i_{\ell (\alpha)}} x_{i_1}^{\alpha _1}\cdots
x_{i_{\ell (\alpha)}}^{\alpha _{\ell (\alpha)}}.$$
This basis is called the basis of \emph{monomial quasisymmetric functions}.
A second basis for $\QQ$, called the basis of \emph{fundamental quasisymmetric
functions}, consists of $F_0=1$ and
$$F_\alpha = \sum _{\beta \preccurlyeq \alpha} M_\beta.$$
A third, recently discovered, basis is the basis of \emph{quasisymmetric Schur
functions}, which consists of $\qs _0=1$ and $\{ \qs_\alpha \} _{\alpha \vDash n}$
that are defined as follows.

\begin{definition}\label{def:QSasF}
For $\alpha \vDash n$ let the \emph{quasisymmetric Schur function} $\qs _\alpha$ be
$$\qs _\alpha = \sum _{\beta \vDash n} d_{\alpha\beta} F_\beta$$
where $d_{\alpha\beta}=$ the number of SCTx $T$ of shape~$\alpha$ and
$\com(T)=\beta$.
\end{definition}

\begin{example} If $n=4$
$$M_{(1,3)}= x_1x_2^3 + x_1x_3^3 + x_2x_3^3 + \cdots\quad F_{(1,3)}= M_{(1,3)}+M_{(1,2,1)}+M_{(1,1,2)}+M_{(1,1,1,1)}$$$$ \qs _{(1,3)}=F_{(1,3)} + F_{(2,2)}$$
from the SCTx
$$\begin{matrix}1\\4&3&2\end{matrix}\qquad \begin{matrix}2\\4&3&1\end{matrix} .$$
\end{example}

The algebra of symmetric functions, $\SSS$, is a graded subalgebra of $\QQ$
$$\SSS := \SSS_0 \oplus \SSS _1 \oplus \cdots$$with a variety of bases, the most renowned of which is the basis consisting of Schur functions, whose connection
with quasisymmetric Schur functions can be described very simply as follows. Let $\lambda$ be a partition, then the \emph{Schur function} $s_\lambda$ is
\begin{equation}\label{eq:SasQS}
s_\lambda = \sum _{\alpha \in \langle \lambda \rangle} \qs _\alpha .
\end{equation}
A more familiar description is the following, which we state more
generally for the case of skew Schur functions, and
which will be crucial in this paper. We will take this
to be our definition of skew Schur functions.

\begin{definition}\cite[Theorem 7.19.7]{ECII}\label{def:SasF}
For $D$ a skew diagram with $|D|=n$,  
let the \emph{skew Schur function} $s _D$ be
$$s _D = \sum _{\beta \vDash n} d_{D\beta} F_\beta$$ 
where $d_{D\beta}=$ the number of SYTx $T$ of shape~$D$ and $\com(T)=\beta$. 
When $D=\lambda$ is a partition, we call $s_\lambda$ a \emph{Schur function}.
\end{definition}

Another useful result for us will be the expansion of a skew Schur function 
in terms of Schur functions, known as the \emph{Littlewood-Richardson (LR) rule}: 
for $D$ a skew diagram, 
the expansion of the skew Schur function $s_D$  in the basis of Schur functions
is given by
$$s_D= \sum c_{D\lambda} s_\lambda $$
where $c_{D\lambda}$ is the number of LR tableaux of shape~$D$ with
$\lambda _1$ ones, $\lambda _2$ twos, etc.,
where $\lambda = (\lambda _1 , \lambda _2 ,\ldots)$.

We are now ready to present our final definition, which will be the focus of our study.

\begin{definition}\label{def:fmf}
Let $G$ be a quasisymmetric function. Then we say
$$G=\sum _\alpha c_\alpha F_\alpha$$
is \emph{\fmf\ }if $c_\alpha=0$ or $1$ for all compositions $\alpha$.
\end{definition}

\begin{remark}\label{rem:fmf-comb}
Note that in the case of quasisymmetric Schur functions
and skew Schur functions we may
rephrase this condition combinatorially in the following way.
By Definition~\ref{def:QSasF},
a quasisymmetric Schur function $\qs _\alpha$ is
\fmf\ if and only if the SCTx of shape~$\alpha$
all have different descent sets,
and by Definition~\ref{def:SasF}, a skew Schur function $s_D$ is
\fmf\ if and only if the SYTx of shape~$D$ all have different
descent sets.
\end{remark}

\section{Multiplicity free Schur and skew Schur functions}\label{sec:FmfSchur}

In this section we determine necessary and sufficient conditions for when
$$s_D = \sum _\beta d_{D\beta} F _\beta$$
is \fmf\ for some (skew) diagram~$D$.
We begin by resolving the situation when $D$  is a partition. However, we first reduce the number of cases we need to consider with the following lemma.

\begin{lemma}\label{lem:transposes}
Let $\lambda$ be a partition. Then $s_\lambda$ is \fmf\ if and only if $s_{\lambda ^t}$ is \fmf.
\end{lemma}

\begin{proof} Using the involution $\omega$ on quasisymmetric functions that satisfies $\omega(F_{\alpha ^\ast})=F_{\hat{\alpha}}$ and $\omega(s _\lambda)=s_{\lambda ^t}$ \cite[Section 5]{Ehrenborg}, if $s_{\lambda} = \sum _\alpha d_{\lambda \alpha} F _\alpha$ is \fmf\ then
$$s_{\lambda ^t} = \omega(s _\lambda) = \omega (\sum _\alpha d_{\lambda \alpha} F _\alpha) = \sum _\alpha d_{\lambda \alpha} F_{\hat{\alpha}^\ast}$$is \fmf. The converse then holds since transposing is an involution.
\end{proof}

We now identify two families of Schur functions that are \fmf\ by giving explicit formulae for them.
Below, and elsewhere when the meaning is clear from the context, we will also index quasisymmetric functions  by the sets $\set(\alpha)$ instead of
the compositions~$\alpha$, when the formulae can be stated more succinctly this way.

\begin{lemma}\label{lem:2rowhook}
\begin{enumerate}[(i)]
\item For $n \geq 1, 0\leq k \leq n-1$,
$$s _{(n-k , 1^k)} = \sum _{{R\subseteq [n-1]} \atop {|R|=k}}F_R \:.$$
\item For $n \geq 4$,
$$s _{(n-2 , 2)} =
\sum _{i=2}^{n-2} F_{\{i\}} + \sum _{j=3}^{n-1} \sum _{i=1}^{j-2} F_{\{ i,j\}} \:.$$
\end{enumerate}
\end{lemma}

\begin{proof} 
For the first part, observe that for an SYT of shape~$(n-k, 1^k)$, with
$t_1,\ldots,t_k$  in the first column but not the first row, its descent set is the
$k$-element set $R=\{t_1-1,\ldots,t_k-1\} \subseteq [n-1]$.
Since the entries in an SYT of  shape $(n-k, 1^k)$ increase along the first row and first column, for any $k$-element set~$R\subseteq [n-1]$ we get a unique
SYT of shape $(n-k, 1^k)$ with descent set $R$.
The result now follows from Definition~\ref{def:SasF}.

For the second part, observe that an SYT of shape $(n-2,2)$ with entries $i+1$, $i+2$ in the second row has only one descent $i$, and this can only take values
in $[2,\ldots,n-2]$.
Meanwhile an SYT of this shape with entries $i+1$, $j+1$ in the second row with $i+1<j$
has  descent set $\{i,j\}$, but this only occurs for
$i\in [1,\ldots,n-3]$, $j\in [i+2,\ldots.n-1]$. Thus, note that any SYT
of shape $(n-2,2)$ is determined by its descent set.
The result now follows from Definition~\ref{def:SasF}.
\end{proof}

\begin{theorem}\label{the:SFmf}
Let $\lambda$ be a partition of $n$. Then $s_\lambda$ is \fmf\ if and only if $\lambda$ or $\lambda ^t$ is one of
\begin{enumerate}[(i)]
\item $(3,3)$ if $n=6$,
\item $(4,4)$ if $n=8$,
\item $(n-2, 2)$ if $n \geq 4$,
\item $(n-k, 1^k)$ if $n \geq 1$ and $0\leq k \leq n-1$.
\end{enumerate}
\end{theorem}

\begin{proof}
Note that by direct computation $s_{(3,3)}$ and $s_{(4,4)}$ are \fmf. Thus by  Lemma~\ref{lem:transposes} and Lemma~\ref{lem:2rowhook}, the Schur functions listed are \fmf. Now all that remains is to establish that all other Schur functions are not \fmf. To do this we first note that if $\lambda$ (or $\lambda ^t$) is not one of those listed then the diagram $\lambda$ (or $\lambda ^t$) contains $(3,2,1)$ or $(4,3)$.

Observe that $s_{(3,2,1)}$ is not \fmf\ since it contains
$F_{\{2,4\}}$ with multiplicity two, arising from the SYTx
\begin{equation}\label{exp:2x222}\begin{matrix}1&2&4\\3&6\\5\end{matrix}\qquad \begin{matrix}1&2&6\\3&4\\5\end{matrix}\ .\end{equation}
Similarly $s_{(4,3)}$ is not \fmf\ since it contains $F_{\{2,5\}}$ with multiplicity two, arising from the SYTx
\begin{equation}\label{exp:2x232}
\begin{matrix}1&2&5&7\\3&4&6\end{matrix}\qquad \begin{matrix}1&2&4&5\\3&6&7\end{matrix}
\ .\end{equation} 
Now, if $\lambda$ strictly contains $(3,2,1)$, we are either in the case
$\lambda = (3,3,1)$ where we have the SYTx
$$\begin{matrix}1&2&4\\3&5&6\\7\end{matrix}\qquad \begin{matrix}1&2&6\\3&4&7\\5\end{matrix}\ ,$$
or we can extend the SYTx in \eqref{exp:2x222}
to SYTx of shape~$\lambda$  with the same descent set
by placing $7$ in the top, third or fourth row, and then
extend these by filling the remaining cells in both diagrams identically.
Thus by Definition~\ref{def:SasF}, $s_\lambda$ is not \fmf .

On the other hand, if $\lambda\neq (4,4)$ strictly contains $(4,3)$,
then we can extend the SYTx in \eqref{exp:2x232} to SYTx of shape~$\lambda$
with the same descent set by placing $8$ in the top or third row,
and a further identical extension of both tableaux yields
again that $s_\lambda$ is not \fmf .

If $\lambda ^t \neq (4,4)$ contains $(4,3)$
then similarly $s_{\lambda ^t}$ is not \fmf ,
and $s_\lambda$ is not \fmf , by Lemma~\ref{lem:transposes}. 
\end{proof}

Combining this theorem with the LR rule yields the following more general result.

\begin{theorem}\label{the:skewFmf} 
Let $D$ be a skew diagram of size $n$. Then $s_D$ is \fmf\ if and only if
up to transpose, $D$ or $D^\circ$ is one of
\begin{enumerate}[(i)]
\item $(3,3)$ if $n=6$,
\item $(4,4)$ if $n=8$,
\item $(n-2,2)$ if $n\geq 4$,
\item $(n-k, 1^k)$ if $n\geq 1$ and $0\leq k \leq n-1$,
\item $(n-k) \oplus (1^k)$, $0<k<n$.
\end{enumerate}
Thus this is a complete list of skew diagrams $D$ such that all SYTx
of shape~$D$ have distinct descent sets.
\end{theorem}

\begin{proof} 
It is an easy consequence of the LR rule
that if $R_D$ is the partition consisting
of the row lengths of $D$ listed in weakly decreasing order, and $C_D$ is the partition
consisting of the column lengths of $D$ listed in weakly decreasing order, then $s_D$
contains both $s_{R_D}$ and $s_{(C_D) ^t}$ as summands.

More precisely, consider the SSYT where each column $c$ is filled with $1, 2, \ldots , \ell_c$ and $\ell_c$ is the number of cells in column~$c$. This is clearly an LR tableau contributing~1 to the coefficient of $s_{(C_D) ^t}$. Meanwhile, to obtain the term $s_{R_D}$ consider the SSYT formed as follows. Fill the rightmost cell of the highest not completely filled row with a~1. Then fill the rightmost cell of the second highest  not completely filled row with a~2. Continue  until the rightmost cell of the lowest  not completely filled row is filled. Repeat with remaining unfilled cells  until an SSYT is produced, which is easily checked to be an LR tableau contributing~1 to the coefficient of $s_{R_D}$. 

It also follows from the LR rule
that $s_D=s_{D^\circ}$. Therefore, by Theorem~\ref{the:SFmf} it is sufficient to check
only those skew Schur functions $s_D$ such that $R_D$ is one of the partitions listed
there but $D$ or $D^\circ$ is not one of the diagrams listed therein.

\begin{description}
\item[Case $R_D= (n-2,2)\ n\geq 4 $] Note that $F_{(2,n-2)}$ appears with multiplicity.
\item[Case $R_D=(4^2)$] Using the LR rule this contains  $s_{(5,3)}$ as a term in the Schur function expansion of $s_D$, which is not \fmf.
\item[Case $R_D= (2^4)$] Contains  either $s_{(3,2,2,1)}$ or $s_{(3,3,1,1)}$ as a term in the Schur function expansion of $s_D$, which is not \fmf.
\item[Case $R_D= (3^2)$] Contains  $s_{(4,2)}+s_{(3,3)}$ as a term in the Schur function expansion of~$s_D$, so by the proof of Theorem~\ref{the:SFmf} and Lemma~\ref{lem:2rowhook} we have that $F_{(3^2)}$ appears with multiplicity.
\item[Case $R_D= (2^3)$] Contains $s_{(3,2,1)}$ as a term in the Schur function expansion of~$s_D$, which is not \fmf.
\item[Case $R_D= (2^2,1^{(n-4)}), n\geq 5$] Note that $F_{(1,2^2,1^{n-5})}$ appears with multiplicity.
\item[Case $R_D= (n-k,1^k), n > k$] If $D=(n-k) \oplus (1^k)$ then $s_D = s_{(n-k+1, 1^{k-1})}+s_{(n-k, 1^k)}$. Otherwise if $n=3$ then $F_{(2,1)}$ appears with multiplicity. If $n\geq 4$ then $F_{(1,2,1) \cdot \beta}$ for $\beta \vDash n-4$ will appear with multiplicity.
\end{description}
This concludes the proof.
\end{proof}

\section{Multiplicity free quasisymmetric Schur  functions with one or two terms}\label{sec:QSFmf12terms}
A natural avenue to pursue is to establish when quasisymmetric Schur functions are \fmf. Classifying when quasisymmetric (skew) Schur functions are \fmf\ seems more subtle than the symmetric situation, however, we can classify certain special cases. To this end we begin by classifying all quasisymmetric Schur functions with one or two terms that are \fmf. This task is simplified by the following lemma. For simplicity of exposition we denote the shape of an SCT $T$ by $\sh(T)$.

\begin{lemma}\label{lem:reduction}
Let $\alpha = (\alpha _1 ,\ldots , \alpha _k)$ be a composition.
\begin{enumerate}[(i)]
\item There exists a unique SCT $T$ with $\sh(T)=\alpha$ and $\com(T)=\alpha$.
\item If  $(\alpha _i , \alpha _{i+1})\not\in \{(m,1), (1,2) \mid  m\geq 1\}$ for some $i$, then there exists an SCT $T'$ with $\sh(T')=\alpha$ and $\com(T')\neq\alpha$.
\item If $(\alpha _i , \alpha _{i+1})\not\in \{(m,1), (1,2), (2,2), (2,3), (1,3) \mid m\geq 1\}$ for some $i$, then there exist at least two SCTx $T', T''$, with $\sh(T')=\sh(T'')=\alpha$ and $\com(T'), \com(T'')\neq\alpha$ and $\com(T')\neq \com(T'')$.
\end{enumerate}
\end{lemma}

\begin{proof}
\begin{enumerate}[(i)]
\item This follows immediately by considering the canonical filling of $\alpha$,
which is the unique SCT $T$ with $\sh(T)=\alpha$ and $\com(T)=\alpha$.
\item Now consider a composition diagram $\alpha$ such that there exists a pair
of rows with 
$(\alpha _i , \alpha _{i+1})\not\in \{(m,1), (1,2)\}$. 
Then either $\alpha _i=1, \alpha _{i+1}\geq 3$, or $\alpha _i, \alpha _{i+1}\geq 2$.

We set $x=\sum _{j=1} ^{i+1} \alpha _j$ and describe the filling of rows $i,i+1$,
while the remaining rows are row filled.

If  $\alpha _i=1, \alpha _{i+1}\geq 3$,
we consider the SCT $T'$ with rows $i, {i+1}$ filled
$$\begin{matrix}x-2\\x&x-1&x-3&\cdots&x-\alpha _{i+1}\end{matrix}\:.$$
Now assume $\alpha _i\geq 2$; if $\alpha_{i+1}=2$, we consider
$$\begin{matrix}x-1&x-2&x-4&\cdots &x -\alpha _i -1\\
x&x-3\end{matrix}$$
or otherwise, if $\alpha_{i+1}> 2$ we consider 
$$\begin{matrix}x-2&x-\alpha _{i+1}-1&\cdots&x-\alpha _{i+1}-\alpha _{i}+1\\x&x-1&x-3&\cdots&x-\alpha _{i+1}\end{matrix}\:.$$
Then in all cases we have $\sh(T')=\alpha$ and $\com(T')\neq\alpha$.
\item
Now consider a composition diagram $\alpha$ such that there exists
a pair of rows with 
$$(\alpha _i , \alpha _{i+1})\not\in \{(m,1), (1,2), (2,2), (2,3), (1,3)\}.$$
Then either 
 $\alpha _i=1, \alpha _{i+1}\geq 4$, or
 $\alpha _i=2, \alpha _{i+1}\geq 4$, or
 $\alpha _i\geq 3, \alpha _{i+1}\geq 2$.

Again, we set $x=\sum _{j=1} ^{i+1} \alpha _j$ and focus on the filling
of rows $i,i+1$.

First assume $\alpha _i=1, \alpha _{i+1}\geq 4$.
We consider the SCTx $T', T''$, with rows $i, i+1$ respectively filled
$$\begin{matrix}x-2\\x&x-1&x-3&\cdots&x-\alpha _{i+1}\end{matrix}$$
and
$$\begin{matrix}x-3\\x&x-1&x-2&\cdots&x-\alpha _{i+1}\end{matrix}\ .$$
In the case  $\alpha _i=2, \alpha _{i+1}\geq 4$,  consider
$$\begin{matrix}x-2&x-4\\x&x-1&x-3&x-5&\cdots&x-\alpha _{i+1}-1\end{matrix}$$and $$\begin{matrix}x-3&x-4\\x&x-1&x-2&x-5&\cdots&x-\alpha _{i+1}-1\end{matrix}\ .$$
Finally, assume  $\alpha _i\geq 3$; if $\alpha _{i+1}=2$
we consider
$$\begin{matrix}x-1&x-2&x-4&\cdots&x-\alpha _{i}-1\\x&x-3\\\end{matrix}$$
and
$$\begin{matrix}x-1&x-2&x-3&\cdots&x-\alpha _{i}-1\\x&x-4\end{matrix}\ ,$$
or otherwise, if $\alpha _{i+1} > 2$, consider 
$$\begin{matrix}x-3&x-4&\cdots&\\x&x-1&x-2&\cdots\end{matrix}$$
and
$$\begin{matrix}x-2&x-4&\cdots&\\x&x-1&x-3&\cdots\end{matrix}$$
and the remaining cells of these rows are filled identically.
In all cases the remaining rows are row filled.
Then in all cases we have $\sh(T')=\sh(T'')=\alpha$
with $\com(T'), \com(T'')\neq\alpha$ and $\com(T')\neq \com(T'')$.
\end{enumerate}
\end{proof}

Note that the first part of Lemma~\ref{lem:reduction} says that $F_\alpha$ is always a summand of $\qs _\alpha$ with coefficient~1, and was also established in \cite[Lemma 5.4]{QS}.
\medskip

The following result implies that being \fmf\ is inherited in certain cases.

\begin{lemma}\label{lem:1or2add}
Let $\alpha$ be a composition and
$\qs _\alpha = \sum _\beta d_{\alpha\beta} F_\beta$. Then
$$\qs _{\alpha\cdot (1)} = \sum _\beta d_{\alpha\beta} F_{\beta\cdot (1)}$$and
$$\qs _{\alpha\cdot (1,2)} = \sum _\beta d_{\alpha\beta} F_{\beta\cdot (1,2)}.$$
\end{lemma}

\begin{proof} Note that 
there exists a  bijection between SCTx of shape~$\alpha$ and $\alpha \cdot (1)$: 
simply append (for the inverse, remove) the cell containing $|\alpha|+1$.
Similarly, there exists a bijection  between SCTx of shape~$\alpha$ and 
$\alpha \cdot (1,2)$: simply append (for the inverse, remove) the cells 
containing $|\alpha|+1, |\alpha|+2, |\alpha|+3$ by row filling the final two rows.
\end{proof}

This motivates us to define the following special set of compositions, which also contains the empty composition:
$$C_2=\{ (1^{e_1},2,1^{e_2},\ldots ,2 ,1^{e_k}) \mid  k\in \N_0,
e_i\in \N \mbox{ for } i\in [k-1], e_k \in \N_0\}\:.$$

\begin{corollary}\label{cor:1or2add}
Let $\alpha$ be a composition and $\qs _\alpha = \sum _\beta d_{\alpha\beta} F_\beta$. Then for any $\gamma \in C_2$
$$\qs _{\alpha\cdot \gamma} = \sum _\beta d_{\alpha\beta} F_{\beta\cdot \gamma} \:.$$
\end{corollary}

We define the following set of compositions $C_2'$, which is a subset of $C_2$:
$$C_2'=\{ (1^{e_1},2,1^{e_2},\ldots ,2 ,1^{e_{k-1}},2) \mid  k-1\in \N,
e_i\in \N \mbox{ for } i\in [k-1]\}\:.$$

When $\qs _\alpha = \sum _\beta d_{\alpha \beta}F_\beta$, the
number $|\{\beta \mid d_{\alpha \beta}\neq 0\}|$
is the  {\em number of $F$-components} of $\qs_\alpha$.

\begin{theorem}\label{the:QS12term}
Let $\alpha$ be a composition.
\begin{enumerate}[(i)]
\item
$\qs_\alpha$ has only one $F$-component, and moreover
$\qs _\alpha = F_\alpha$,
if and only if $\alpha = (m)\cdot \gamma$ for some $\gamma\in C_2$,
$m\in \N_0$. (Here $m=0$ should be understood as non-appearing in the composition.)
\item
$\qs_\alpha$ has two $F$-components, and moreover
$\qs _\alpha = F_\alpha+F_\beta$ where $\beta\neq \alpha$,
if and only if
\begin{enumerate}[(I)]
\item $\alpha = (m)\cdot \gamma' \cdot (2) \cdot \gamma$ and  $m\in \N_0$,
$\gamma'\in C_2'$, $\gamma\in C_2$ or
\item $\alpha = (1,3) \cdot \gamma$, $\gamma\in C_2$ or
\item $\alpha = (m)\cdot \gamma' \cdot (3) \cdot \gamma$ and  $m\in \N_0$,
$\gamma'\in C_2'$, $\gamma\in C_2$.
\end{enumerate}
\end{enumerate}
In particular, for all the compositions listed above, $\qs_\alpha$
is \fmf.
\end{theorem}

\begin{proof}
The first part follows from Lemma~\ref{lem:reduction}
and Corollary~\ref{cor:1or2add}, and alternatively is proved directly in \cite[Corollary 6.9]{QS}.

For the second part, by considering all consecutive pairs of parts in $\alpha$,
we immediately deduce that if $\alpha$ is not one of the compositions listed in the theorem, then $\qs _\alpha$  consists of more than two different terms $F_\beta$  by the SCTx created in the proof of Lemma~\ref{lem:reduction},
unless we have three consecutive rows $i,i+1,i+2$ 
of length $m,1,3$ for some $m\geq 1$; or $2,2,2$;  or $2,2,3$.

In the case  $m,1,3$ for some $m\geq 1$ we can obtain three fillings for these rows:
$$\begin{matrix}
x&x-1&x-2&\cdots\\
x+1\\
x+4&x+3&x+2
\end{matrix}\quad
\begin{matrix}
x&x-1&x-2&\cdots\\
x+2\\
x+4&x+3&x+1
\end{matrix}\quad
\begin{matrix}
x+1&x-1&x-2&\cdots\\
x+2\\
x+4&x+3&x
\end{matrix}$$where $x=\sum _{j=1} ^{i}\alpha _j$; when $m<3$, the first row is accordingly shortened.

Similarly, in the case  $2,2,2$ we can obtain three fillings for these rows:
$$\begin{matrix}
x&x-1\\
x+2&x+1\\
x+4&x+3
\end{matrix}\qquad
\begin{matrix}
x&x-1\\
x+3&x+2\\
x+4&x+1
\end{matrix}\qquad
\begin{matrix}
x+2&x+1\\
x+3&x\\
x+4&x-1
\end{matrix}$$where $x=\sum _{j=1} ^{i}\alpha _j$.

Also, in the case  $2,2,3$ we can obtain three fillings for these rows:
$$\begin{matrix}
x&x-1\\
x+2&x+1\\
x+5&x+4&x+3
\end{matrix}\qquad
\begin{matrix}
x&x-1\\
x+3&x+1\\
x+5&x+4&x+2
\end{matrix}\qquad
\begin{matrix}
x+1&x\\
x+3&x-1\\
x+5&x+4&x+2
\end{matrix}$$where $x=\sum _{j=1} ^{i}\alpha _j$.

We still have to check that the $\qs _\alpha$ for the $\alpha$
listed above expand into exactly two fundamental quasisymmetric functions. 
Note that by the cover
relations in $\LC$, SCTx of shape listed in the  cases above will be created almost
uniquely by being row filled,
apart from the rows corresponding to
$(\alpha _i, \alpha _{i+1})=(2,2), (1,3),(2,3)$.
The cover relations on $\LC$ yield that there are exactly two ways to fill these rows, respectively, where $x=\sum _{j=1} ^{i+1}\alpha _j$:
$$\begin{matrix} x-2&x-3\\x&x-1\end{matrix}\quad\mbox{and}\quad\begin{matrix} x-1&x-2\\x& x-3\end{matrix}$$

$$\begin{matrix} x-3\\x&x-1&x-2\end{matrix}\quad\mbox{and}\quad\begin{matrix} x-2\\x&x-1&x-3\end{matrix}$$

$$ \begin{matrix} x-3&x-4\\x&x-1&x-2\end{matrix}\quad\mbox{and}\quad\begin{matrix} x-2&x-4\\x&x-1&x-3\end{matrix}\:.$$
\end{proof}

\section{Multiplicity free quasisymmetric Schur functions with two parts}\label{sec:QSFmf12parts}
In this section we continue our classification of \fmf\ quasisymmetric Schur functions $\qs _\alpha$.
The case when $\alpha$ has only one part we classified in the
previous section;
when $\alpha$ has exactly two parts we have the following theorem.

\begin{theorem}\label{the:QSFmf12part}
Let $\alpha$ be a composition of $n$ with two parts.
Then $\qs _\alpha$ is \fmf\ if and only if $\alpha$ is one of
\begin{enumerate}[(i)] 
\item $\alpha = (n-1, 1), (1,n-1)$, for $n \geq 2$,
\item $\alpha = (n-2, 2), (2,n-2)$, for $n \geq 4$,
\item $\alpha = (n-3, 3)$, for $n \geq 6$,
\item $\alpha = (3,4), (4,4), (4,5)$.
\end{enumerate}
\end{theorem}

We will prove this theorem via the following lemmas, although parts of the result   could  also be obtained from Theorem~\ref{the:SFmf} and Lemma~\ref{lem:2rowhook}.

\begin{lemma}\label{lem:2row123}
\begin{enumerate}[(i)]
\item For $n\geq 3$,
\begin{align*}
\qs _{(n-1, 1)}&=F_{\{n-1\}},\\
\qs _{(1, n-1)}&= \sum_{i=1} ^{n-2} F_{\{i\}}.
\end{align*}
\item For $n\geq 5$,
\begin{align*}
\qs _{(n-2, 2)}&=F_{\{n-2\}}+ \sum_{i=1} ^{n-3} F_{\{i, n-1\}},\\
\qs _{(2, n-2)}&= \sum_{i=2} ^{n-3} F_{\{i\}}+\sum _{j=3}^{n-2}\sum_{i=1} ^{j-2} F_{\{i,j\}}.
\end{align*}
\item  For $n\geq 7$, 
\begin{align*}
\qs _{(n-3, 3)}
=F_{\{n-3\}}+ F_{\{n-4, n-2\}} + \sum_{i=1} ^{n-5} F_{\{i, n-2\}}
+ \sum_{i=2} ^{n-4} F_{\{i, n-1\}}+\sum _{j=3}^{n-2}\sum_{{i, j \in [n-3]} \atop {j\neq i+1}}  F_{\{i,j, n-1\}}.
\end{align*} 
\end{enumerate}
\end{lemma}

\begin{proof}
(i) The expression for $\qs _{(n-1, 1)}$ 
is already contained in Theorem~\ref{the:QS12term}.
The expression for $\qs _{(1,n-1)}$ then follows from Lemma~\ref{lem:2rowhook} and the formula for $\qs _{(n-1, 1)}$.

(ii)  Consider 
a composition diagram of shape~$(n-2,2)$.
In order to create an SCT of shape~$\alpha$, if the second row is filled
$$n \ i$$
then the first row is filled with $\{1, \ldots , n-1\}\backslash \{i\}$
in decreasing order from left to right.
Thus, all we need to do is determine which values $i$ can take using the
cover relations  in $\LC$.
This now yields that $i\neq n-2$, and the expression for $\qs _{(n-2,2)}$
immediately follows. The expression for $\qs _{(2,n-2)}$ then follows from
Lemma~\ref{lem:2rowhook} and the formula for $\qs _{(n-2,2)}$.

(iii) Finally,  consider a composition diagram of shape~$(n-3 , 3)$.
In order to create an SCT of this shape, the second row must be filled with
$\{i,j,n\}$ and the first row filled with $\{1, \ldots , n-1\}\backslash \{i,j\}$
in decreasing order from left to right.
For the cover relations of $\LC$ to be satisfied, $\{i,j\}$ must be one of
\begin{enumerate}
\item $\{n-2, n-1 \}$
\item $\{n-3, n-1 \}$
\item $\{i, n-1 \}$ for $1\leq i \leq n-5$
\item $\{i, i+1 \}$ for $1\leq i \leq n-5$
\item $\{i, j \}$ for $1\leq i , j\leq n-3$ and $i<j\neq i+1$
\end{enumerate} which results in the five types of summand listed above, respectively.
\end{proof}

\begin{lemma}\label{lem:3mmult} Let $n\geq 8$. Then $F_{\{ 2,5\}}$ appears in the expansion of $\qs _{(3,n-3)}$ with multiplicity.
\end{lemma}

\begin{proof}
Consider the SCT $T'$ of shape~$(3, n-3)$, whose first row is filled with $\{1,2,5\}$ in decreasing order left to right, and second row is filled with $\{3,4,6,\ldots, n\}$   in decreasing order left to right. Also consider the SCT $T''$ of shape~$(3, n-3)$, whose first row is filled with $\{2,4,5\}$ in decreasing order left to right, and second row is filled with $\{1,3,6,\ldots, n\}$   in decreasing order left to right. Then $\des(T')=\des(T'')=\{2,5\}$ and the result follows. 
\end{proof}

\begin{lemma}\label{lem:m4mult}
Let $n\geq 9$ and $n-m>m\geq 4$. Then $F_{\{ 2,5, n-m+3\}}$ appears in the expansion of $\qs _{(n-m,m)}$ with multiplicity.
\end{lemma}

\begin{proof} As in the proof of Lemma~\ref{lem:3mmult}, we give two sets of integers to fill a row, thus determining two SCTx 
$T', T''$ with $\des(T')=\des(T'')=\{2,5,n-m+3\}$. In this case, the two sets to fill the first row with are
$$\{1,3,6,7,\ldots, n-m+3\}\mbox{ and }\{3,4,6,7,\ldots, n-m+3\}.$$
\end{proof}

\begin{lemma}\label{lem:4mmult}
Let $n\geq 10$ and $n-m>m\geq 4$. Then $F_{\{ 2,5, n-m+2\}}$ appears in the expansion of $\qs _{(m, n-m)}$ with multiplicity.
\end{lemma}

\begin{proof} As in the proof of Lemma~\ref{lem:3mmult} we give two sets of integers to fill a row, thus determining two SCTx 
$T', T''$ with $\des(T')=\des(T'')=\{2,5,n-m+2\}$. If $n-m>m+1$, the two sets to fill the first row with are
$$\{1,2,5, n-2m+6,\ldots, n-m+2\}\mbox{ and }\{2,4,5, n-2m+6,\ldots, n-m+2\}.$$However, if $n-m=m+1$
the two sets to fill the first row with are
$$\{1,2,5, 7,\ldots, n-m+2\}\mbox{ and }\{1, 2,4,5, 8,\ldots, n-m+2\}.$$
\end{proof}

We are now ready to prove Theorem~\ref{the:QSFmf12part}.
\begin{proof}\emph{(of Theorem~\ref{the:QSFmf12part})} Let $\alpha$ be a composition with two parts. By direct calculation we see that $\qs _{(4,5)}$ and $\qs _{(3,4)}$ are \fmf .
Since $s_\lambda = \sum _{\lambda(\alpha)=\lambda}\qs _\alpha$, it follows that $s_{(m,m)}=\qs _{(m,m)}$, and hence $\qs _{(m,m)}$ is \fmf\ if and only if
$(m,m)\in \{(1,1), (2,2), (3,3), (4,4)\}$ by Theorem~\ref{the:SFmf}. The remaining $\qs _\alpha$ stated in the theorem are \fmf\ by Lemma~\ref{lem:2row123}.
For all other $\alpha$, $\qs _\alpha$ is not \fmf\ by Lemmas~\ref{lem:3mmult}, \ref{lem:m4mult}, and \ref{lem:4mmult}.
\end{proof}

\section{Multiplicity free quasisymmetric Schur function families}\label{sec:qsfams}

Refining our results from Section~\ref{sec:FmfSchur}, we can also identify partitions $\lambda$ such that  $\qs _\alpha$ is \fmf\ for all
$\alpha \in \langle\lambda\rangle$. This gives us a classification, in some sense, of Schur functions that are almost \fmf , since
$$s_\lambda = \sum _{\alpha \in \langle\lambda\rangle  }\qs _\alpha .$$
Note that in general it is not true that the quasisymmetric Schur functions
indexed by $\alpha \in \langle\lambda\rangle$ all behave in the same way;
for example, $\qs _{(2,3,3)}$ is \fmf\ while $\qs _{(3,3,2)}$ is not.
Additionally, the classification we want to obtain
yields further quasisymmetric Schur functions that are \fmf .
We accomplish this task by the following theorem, whose proof we devote the rest of this section to.

\begin{theorem}\label{the:Fmffam}
Let $\lambda$ be a partition of $n\in \N$. Then $\qs _\alpha$  is \fmf\
for all $\alpha \in \langle\lambda\rangle $ if and only if $\lambda$ is one of
\begin{enumerate}[(i)]
\item $(n-k, 1^k)$ if $0\leq k \leq n-1$,
\item $(n-2-k, 2, 1^k)$ if $0\leq k\leq n-5$,
\item $(2^a, 1^{n-2a})$ if $2\leq a \leq 4$ and $2a\leq n$,
\item $(3,2,2,1^{n-7})$ if $n\geq 7$,
\item $(3,3), (4,3), (4,4)$.
\end{enumerate}
\end{theorem}

Observe that from the cover relations on $\LC$ we obtain the following result.
\begin{lemma}\label{lem:12exp}
Let $\alpha = (1^{f_1},2^{e_1},1^{f_2},\ldots ,1^{f_{k-1}},2^{e_{k-1}},1^{f_k})$,
where $f_1, f_k \in  \N_0$, $f_2, \ldots, f_{k-1} \in \N$, $e_1,\ldots,e_{k-1} \in \N$.
Then
$$\qs_\alpha=
\sum_{(\gamma_1,\ldots,\gamma_{k-1})} (\prod_i d_{(2^{e_i}) \gamma_i})  
F_{(1^{f_1}) \cdot \gamma_1 \cdot (1^{f_2}) \cdot \gamma_2 \cdot \cdots \cdot (1^{f_{k-1}}) \cdot \gamma_{k-1} \cdot ( 1^{f_k})}\:,$$
where the sum runs over all $(k-1)$-tuples of compositions
$(\gamma_1,\ldots,\gamma_{k-1})$ with $\gamma_i \vDash 2e_i$, $i=1,\ldots,k-1$.
\end{lemma}

\begin{proof}
We begin by considering $\alpha = (1^f, 2^e)$ for $f\in \N_0$, $e\in \N$. Note that in any SCT of shape~$\alpha$ the numbers $1, 2, \ldots , f$ appear in the top $f$ rows in increasing order, and the numbers $f+1, \ldots , f+2e$ appear in the bottom $e$ rows. Otherwise some number  $m>f$ appears in the top $f$ rows and hence one of the numbers $1, 2, \ldots , f$ appears in a row $f+1,\ldots , f+e$, which is impossible by the cover relations on $\LC$. Thus only SCTx with descent
composition $(1^f)\cdot \gamma$, $\gamma \vDash 2e$, can appear, with multiplicity~ $d_{(2^{e}) \gamma}$.

As a consequence, if
$\alpha = (1^{f_1},2^{e_1},1^{f_2},\ldots ,1^{f_{k-1}},2^{e_{k-1}},1^{f_k}) \vDash n$ then an SCT of shape~$\alpha$
can only be constructed by placing $n, n-1 ,\ldots , n-f_k+1$ in increasing order in the bottom $f_k$ rows in the first column,
then placing $n-f_k , \ldots , n-f_k - 2e_{k-1} +1$ in
one of $d _{(2^{e_{k-1}}) \gamma_{k-1}}$
ways in the next $e_{k-1}$ rows and first and second columns,
with $\gamma_{k-1}\vDash 2e_{k-1}$.
Similarly fill the $f_{k-i} + e _{k-i-1}$ rows,
until the final $f_1$ rows are filled uniquely with $1, \ldots , f_1$.

In this way we construct \ $\prod_i d_{(2^{e_i}) \gamma_i}$ \ 
SCTx of shape~$\alpha$ and descent composition
$\gamma=(1^{f_1}) \cdot \gamma_1 \cdot (1^{f_2}) \cdot \gamma_2 \cdot \cdots \cdot (1^{f_{k-1}}) \cdot \gamma_{k-1} \cdot ( 1^{f_k})$,
$\gamma_i \vDash 2e_i$ for $i=1,\ldots,k-1$,
and we are done.
\end{proof}

\begin{corollary}\label{cor:2mf}
Let $\alpha = (1^{f_1},2^{e_1},1^{f_2},\ldots ,1^{f_{k-1}},2^{e_{k-1}},1^{f_k})$
be as above.
Then $\qs_\alpha$ is \fmf\ if and only if $e_i\leq 4$ for all $i$.
\end{corollary}

\begin{proof}
By Theorem~\ref{the:SFmf}, $\qs_{(2^a)}=s_{(2^a)}$ is \fmf\ if and only if
$0\leq a \leq 4$.
Hence the claim follows by Lemma~\ref{lem:12exp}.
\end{proof}

\begin{corollary}\label{cor:2222}
Let  $\lambda=(2^a, 1^{n-2a})$, with  $0\leq a \leq n/2$.
Then $\qs _\alpha$ is \fmf\ for all $\alpha \in \langle  \lambda \rangle $
if and only if $a\leq 4$.
\end{corollary}

\begin{lemma}\label{lem:21m}
Let $\alpha \in \langle (n-2-k, 2, 1^k)\rangle $ for $0\leq k\leq n-5$.
Then $\qs _\alpha$ is \fmf.
\end{lemma}

\begin{proof} 
By Corollary~\ref{cor:1or2add} we may
drop a trailing sequence of ones from $\alpha$; thus
we only have to consider  the following three cases.
\begin{description}
\item[Case $\alpha = (1^{k_1},n-2-k, 1^{k_2}, 2)$, where $k_1+k_2=k, k_2>0$]
By Lemma~\ref{lem:2rowhook},  $s_{(n-2-k, 1^{k_1})}$ is \fmf .
Thus $\qs _{(1^{k_1}, n-2-k)}$  is \fmf\ by Equation~\eqref{eq:SasQS},
and $\qs _\alpha$ is \fmf\ by Corollary~\ref{cor:1or2add}.
\item[Case $\alpha = (1^{k},n-2-k, 2)$] Consider an  SCT $T$
    of shape~$\alpha$.
    Let $r$ be maximal such that $n,n-1,\ldots,n-(r-1)$ are in the last two rows of $T$.
By the cover relations in $\LC$ it follows that $r\geq 4$ and that
the last row is filled by two of the $r$ largest numbers. 
Then $n-r$ is the entry in row $k$, and the SCT is determined by the $k-1$ numbers in $\{1,\ldots,n-r-1\}$ which appear in rows $1$ to $k-1$; the remaining entries appear in row $k+1$. The descent set of the SCT is then given by the entries in the first $k$ rows together with the descent set of the subtableaux of shape $(r-2,2)$ filled by the $r$ largest numbers. As the Schur function $s_{(r-2,2)}$ is \fmf ,  then so is $\qs _\alpha$.
\item[Case $\alpha = (1^{k_1}, 2, 1^{k_2},  n-2-k)$, where $k_1+k_2=k$, $n-2-k>2$]  For $T$ an SCT of shape~$\alpha$,
    for $k_2>0$ (resp. $k_2=0$), the cover relations in $\LC$ imply that $n,n-1$ (resp. $n,n-1, n-2$) must appear in the last row.  Let $X$ be the set of entries in rows $k_1+1,\ldots,k_1+1+k_2$, then
    $X\subset \{k_1+1,\ldots,n-2\}$ (resp. $X\subset \{k_1+1,\ldots,n-3\}$ ) and its two smallest
    elements $z>x$ are the entries in row $k_1+1$.
    Let $X'$ be the set of entries in the first $k_1$ rows;
    note that $\max X' < x = \min X$. The  descent set of $T$ is then
    $X'\cup X\setminus\{x\}$ if $z=x+1$ and $X\cup X'$  otherwise.  Thus the descent set determines $T$, and hence $\qs _\alpha$ is \fmf .
\end{description}
\end{proof}

\begin{lemma}\label{lem:3221}
Let $\alpha \in \langle (3,2, 2, 1^{n-7})\rangle  $ for $n\geq 7$.
Then $\qs _\alpha$ is \fmf.
\end{lemma}
\begin{proof}
Note that $\alpha$ can be written as
$\alpha = \beta \cdot (1^k) \cdot (3) \cdot \gamma$ for some suitable compositions
$\beta \vDash m_1$ where the last part of $\beta$ is 2 (or possibly $\beta=\emptyset$),
and $\gamma \vDash m_2$ such that $m_1+m_2+k+3=n$.
Then by the cover relations on $\LC$ it follows that to make an SCT
of shape~$\alpha$ we
first need to make an SCT $T$ of shape~$(2)\cdot \gamma$.

Then we  create an SCT $T'$ of shape $1^k\cdot (3) \cdot \gamma$ from $T$ by placing
$x\in [k+1]$ in the top row of $T+(k+1)$ and $[k+1]\setminus\{ x\}$ in the first column
in increasing order. Then extend this to an SCT $T''$ of shape~$\alpha$ by appending an
SCT of shape~$\beta$ on top of $T'+m_1$. All SCTx created this way have distinct descent
sets since
$\beta, \gamma \in \langle (2,2,1^{r_1})\rangle   _{r_1\geq 0} \cup \langle
(2,1^{r_2})\rangle_{(r_2\geq 0)}\cup \langle  (1^{r_3})\rangle   _{r_3\geq 0}$.

The only SCTx of shape~$\alpha$ that are not created in this way by extending $(2)\cdot \gamma$ are those when we extend $T+(m_1+k+1)$
$$\begin{matrix}
m_1+1&m_1-1\\
m_1+2\\
\vdots\\
m_1+k+1\\
m_1+k+3&m_1+k+2&m_1
\end{matrix}$$or the last two parts of $\beta$ are both $2$ and we extend $T+(m_1+k+1)$ by one of
$$\begin{matrix}
m_1&m_1-2\\
m_1+1&m_1-3\\
m_1+2\\
\vdots\\
m_1+k+1\\
m_1+k+3&m_1+k+2&m_1-1
\end{matrix} \qquad
\begin{matrix}
m_1-1&m_1-2\\
m_1+1&m_1-3\\
m_1+2\\
\vdots\\
m_1+k+1\\
m_1+k+3&m_1+k+2&m_1
\end{matrix} .$$
Note that in each of these additional cases the descent sets are unique and differ from those created earlier since now $\{ m_1+1,\ldots , m_1+k+1\}$ are all descents. Hence  $\qs _\alpha$ is \fmf.
\end{proof}

\begin{lemma}\label{lem:leq}
If $\qs _\alpha = \sum _\beta d_{\alpha\beta} F_\beta$ then
$$d_{\alpha\beta}\leq d_{(\alpha\cdot \gamma)(\beta \cdot\gamma)}$$and
$$d_{\alpha\beta}\leq d_{( \gamma \cdot \alpha)(\gamma \cdot \beta)}$$for any composition $\gamma$.
\end{lemma}

\begin{proof} Consider an SCT of shape~$\alpha$ and descent set $\set(\beta)$ and extend this to an SCT of shape $\alpha \cdot \gamma$ and descent set $\set(\beta \cdot \gamma)$ by letting the additional rows be row filled. The first result now follows. The second result follows by a similar argument.
\end{proof}

We can now prove the main result of this section.
\begin{proof}\emph{(of Theorem~\ref{the:Fmffam})}
By Lemmas~\ref{lem:2rowhook}, \ref{lem:21m}, \ref{lem:3221}, Corollary~\ref{cor:2222},
and direct computation, we know that the families listed in
Theorem~\ref{the:Fmffam} are \fmf.
All we need to do is show that no other families are \fmf.
In order to do this we identify family members that are not \fmf\ by noting 
certain 
quasisymmetric Schur functions $\qs _\beta$ that are not \fmf\ and then applying Lemma~\ref{lem:leq}. We sort the families $\langle \lambda \rangle  $ for partitions $\lambda=(\lambda _1, \lambda _2 , \ldots , \lambda _{\ell(\lambda)})$ by width $w(\lambda)=\lambda _1$.
\begin{description}
\item[Case $\lambda _1 = m\geq 5$ and $\lambda _2\geq 3$] Note that $\qs_\beta$ for $\ell(\beta)=2$ and $\beta \not\in \{(m, 3), (4,5)\}$ is not \fmf\ by Theorem~\ref{the:QSFmf12part}.
\item[Case $\lambda _1=m\geq 4$ and $\lambda_2=\lambda_3=2$] Note that $\qs _{(2,2,m)}$ is not \fmf \ by the following two SCTx.
$$\begin{matrix}
3&1\\
5&4\\
m&m-1&\cdots&6&2
\end{matrix}\qquad
\begin{matrix}
3&2\\
5&1\\
m&m-1&\cdots&6&4
\end{matrix} .$$
\item[Case $\lambda _1=4, \lambda _2=3, \lambda _3 \geq 1$] Note that $\qs _\beta$ for $\beta \in \{ (1,4,3), (2,4,3), (3,4,3)\}$ is not \fmf.
\item[Case $\lambda _1=\lambda _2=4, \lambda _3 \geq 1$] Note that $\qs _\beta$ for $\beta \in \{ (1,4,4), (2,4,4), (3,4,4), (4,4,4)\}$ is not \fmf.
\item[Case $\lambda _1=\lambda _2=3, \lambda _3 \geq 1$] Note that $\qs _\beta$ for $\beta \in \{ (1,3,3), (3,3,2), (3,3,3)\}$ is not \fmf.
\item[Case $\lambda _1=3, \lambda _2=\lambda _3 =\lambda _4=2$] In this case $\qs _{(2,2,3,2)}$ is not \fmf.
\item[Case $\lambda _1=\lambda _2=\lambda _3 =\lambda _4=\lambda _5=2$] In this case $\qs _{(2^5)}$ is not \fmf.
\end{description}
\end{proof}

\section{Further directions}

We conclude by outlining three possible directions to pursue.

As discussed in the beginning of Section~\ref{sec:QSFmf12terms}, one natural goal would be to generalize Theorems~\ref{the:QS12term}, \ref{the:QSFmf12part}, and \ref{the:Fmffam} to obtain a classification  of \fmf\ quasisymmetric Schur functions, and then generalize this classification to encompass \emph{skew} quasisymmetric Schur functions.

Related to this latter classification, a second goal could be to determine when a skew quasisymmetric Schur function is multiplicity free when expanded as a linear combination of quasisymmetric Schur functions.

Finally, it is possible to return to our original focus of symmetric functions. In particular, using the expansion of integral Macdonald polynomials in terms of fundamental quasisymmetric functions \cite[Equation (7.11)]{QS}, one could investigate when Macdonald and Hall-Littlewood polynomials are \fmf .


\begin{thebibliography}{130}
\bibliographystyle{abbrv}

\bibitem{Bessenrodt} C.\  Bessenrodt,  On multiplicity-free products of Schur $P$-functions, {\it Ann.\ Combin.} 6 (2002) 119--124.

\bibitem{SQS} C.\ Bessenrodt, K.\ Luoto and S.\ van Willigenburg,
Skew quasisymmetric Schur functions and noncommutative Schur functions, {\it Adv.\ Math.} 226 (2011) 4492--4532.

\bibitem{BHvW} L.\ Billera, S.\ Hsiao and S.\ van Willigenburg, Peak quasisymmetric functions and Eulerian enumeration, {\it Adv.\ Math.} 176  (2003) 248--276.

\bibitem{Ehrenborg} R.\ Ehrenborg, On posets and Hopf algebras, {\it Adv.\ Math.} 119
  (1996) 1--25.

\bibitem{Gessel} I.\ Gessel, Multipartite P-partitions and inner products of skew
  Schur functions,
Combinatorics and algebra, Proc. Conf., Boulder/Colo. 1983,
{\it Contemp.\  Math.} 34 (1984) 289--301.

\bibitem{Gutschwager}
C.\ Gutschwager, On multiplicity-free skew characters and the Schubert calculus,
{\it Ann.\ Comb.}  14 (2010) 339--353.

\bibitem{QS}
J.\ Haglund, K.\ Luoto, S.\ Mason and S.\ van Willigenburg,
Quasisymmetric Schur functions,
{\it J.\ Combin.\ Theory Ser.\ A}  118 (2011)  463--490.

\bibitem{LRQS} J.\ Haglund, K.\ Luoto, S.\ Mason and S.\ van Willigenburg,
 Refinements of the Littlewood-Richardson rule,
 {\it Trans.\ Amer.\ Math.\ Soc.} 363 (2011) 1665--1686.

\bibitem{Hersh-Hsiao} P.\ Hersh and S.\ Hsiao, Random walks on quasisymmetric functions,
{\it Adv.\ Math.} 222 (2009) 782--808.

\bibitem{Hivert} F.\ Hivert, {Hecke algebras, difference operators, and
  quasi-symmetric functions}, {\it Adv.\ Math.} 155 (2000) 181--238.

 \bibitem{LLMS} { T.\ Lam, L.\ Lapointe, J.\ Morse and M.\ Shimozono},
{\em Affine insertion and Pieri rules for the affine Grassmannian},
Memoirs of the American Mathematical Society, Providence,
2010.

\bibitem{Lauve-Mason} A.\ Lauve and S.\ Mason, QSym over Sym has a stable basis,
{\it J.\ Combin.\ Theory Ser.\ A} to appear.

\bibitem{MR} C.\ Malvenuto and C.\ Reutenauer, Duality between quasi-symmetric
functions and the Solomon descent algebra, {\it J.\  Algebra} 177
(1995) 967--982.

\bibitem{ShawvW} K.\ Shaw and S.\ van Willigenburg, Multiplicity free expansions of Schur $P$-functions,
{\it Ann.\ Comb.} 11 (2007) 69--77.

\bibitem{ECII}R.\ Stanley,  {\em {Enumerative
  Combinatorics vol.~2}}, Cambridge University Press, Cambridge, 1999.

\bibitem{Stembridge}J.\ Stembridge, Multiplicity-free products of Schur functions,
{\it Ann.\ Combin.} 5 (2001) 113--121.

\bibitem{ThomasYong}
H.\ Thomas and A.\ Yong,
Multiplicity free Schubert calculus,
{\it Canad.\ Math.\ Bull.}  53 (2010) 171--186.

\end{thebibliography}
\end{document}